\documentclass[reqno,11pt]{article}
\usepackage{a4wide,xcolor,eucal,enumerate,mathrsfs}
\usepackage[normalem]{ulem}
\usepackage[pdfborder={0 0 0}]{hyperref}  

\usepackage{amsmath,amssymb,epsfig,amsthm}
\usepackage[latin1]{inputenc}

\numberwithin{equation}{section}

\newtheorem{theorem}{Theorem}[section]

\newtheorem{corollary}[theorem]{Corollary}
\newtheorem{lemma}[theorem]{Lemma}
\newtheorem{proposition}[theorem]{Proposition}

\newtheorem{definition}[theorem]{Definition}

\newtheorem{remark}[theorem]{Remark}

\newcommand{\ppi}{{\mbox{\boldmath$\pi$}}}
\newcommand{\ggamma}{{\mbox{\boldmath$\gamma$}}}
\newcommand{\sggamma}{{\mbox{\scriptsize\boldmath$\gamma$}}}
\newcommand{\diam}{\mathop{\rm diam}\nolimits}
\newcommand{\supp}{\mathop{\rm supp}\nolimits}
\renewcommand{\d}{{\mathrm d}}

\newcommand{\R}{\mathbb{R}}

\newcommand{\mm}{\mathfrak m}
\newcommand{\nn}{\mathfrak n}
\newcommand{\sfd}{{\sf d}}
\newcommand{\pro}{\mathscr P}
\newcommand{\prob}[1]{\mathscr P(#1)}
\newcommand{\probt}[1]{\mathscr P_2(#1)}

\newcommand{\geo}{{\rm{Geo}}}
\newcommand{\e}{{\rm{e}}}
\newcommand{\gopt}{{\rm{OptGeo}}}

\newcommand{\fe}{f_\varepsilon}
\newcommand{\ue}{u_\varepsilon}

\newcommand{\CD}{{\sf CD}}
\newcommand{\RCD}{{\sf RCD}}
\newcommand{\BE}{{\sf BE}}
\newcommand{\Ch}{{\sf Ch}}
\renewcommand{\H}{{\sf H}}
\newcommand{\M}{{\mathbb M}}
\newcommand{\LIP}{{\sf LIP}}
\newcommand{\calW}{{\mathcal S}^2}

\title{Li-Yau and Harnack type inequalities in $\RCD^*(K,N)$ metric measure spaces}
\begin{document}
\author{Nicola Garofalo
\thanks{First author supported in part by NSF Grant DMS-1001317. \emph{Address}: Dipartimento di Ingegneria Civile, Edile e Ambientale (DICEA),  Universit\`a di Padova, 35131 Padova, Italy. \emph{Email}: nicola.garofalo@unipd.it}
      \and
   Andrea Mondino
   \thanks{Second author supported by an ETH fellowship. Part of this paper was written when he was supported by the ERC grant GeMeTheNES directed by Prof. Luigi Ambrosio. \emph{Address}: Mathematics Department, ETH, Zurich. \emph{Email}:andrea.mondino@math.ethz.ch}}

%\author{Nicola Garofalo \& Andrea Mondino}
%\address{}
%\email[Nicola Garofalo]{nicola.garofalo@unipd.it}

%\author{Andrea Mondino}

%\address{ETH, Zurich, Switzerland} 

%\email[Andrea Mondino]{andrea.mondino@math.ethz.ch}

\maketitle

\begin{abstract}
Metric measure spaces satisfying the reduced curvature-dimension condition $\CD^*(K,N)$ and where the heat flow is linear are called $\RCD^*(K,N)$-spaces. This class of non smooth spaces contains Gromov-Hausdorff limits of Riemannian manifolds with Ricci curvature bounded below by $K$ and dimension bounded above by $N$. We prove that in $\RCD^*(K,N)$-spaces the following properties of the heat flow hold true: a Li-Yau type inequality, a Bakry-Qian inequality,  the Harnack inequality. 
\end{abstract}


%\textit{Mathematics Subject Classification:} 51F99-53B99 ??.

\tableofcontents

\section{Introduction}
Given a smooth $N$-dimensional Riemannian manifold with nonnegative Ricci curvature, a  celebrated
inequality of Li and Yau \cite{Li-Yau86} states that, for every smooth nonnegative function $f$,  one has
\begin{equation}\label{eq:ClassicLiYau}
\Delta(\log \H_t f)\geq \frac{N}{2t},
\end{equation}
where $\H_t=e^{t\Delta}$ indicates the heat semigroup associated to the Laplace-Beltrami operator $\Delta$ (strictly speaking, Li and Yau proved a stronger result, since \eqref{eq:ClassicLiYau} is valid for all positive solutions of the heat equation, not just for $u = \H_t f$). One of the main objectives of this paper is to establish Li-Yau type inequalities in non smooth spaces with Ricci curvature lower bounds. Let us briefly introduce the framework. 

Throughout the paper $(X,\sfd,\mm)$ indicates a metric measure space, m.m.s. for short; i.e., $(X,\sfd)$ is a complete and separable metric space (possibly non compact) and $\mm$ is a probability measure on it (in the setting of smooth Riemannian manifolds $\mm$ corresponds to the volume measure multiplicated by a suitable Gaussian and $\sfd$ is the usual Riemannian distance). 

In this framework, using tools of optimal transportation, Lott-Villani \cite{Lott-Villani09} and Sturm \cite{Sturm06I}-\cite{Sturm06II} detected the class of the so called $\CD(K,N)$-spaces having Ricci curvature bounded below by $K\in \R$ and dimension bounded above by $N\in[1,\infty]$; this notion is compatible with the classical one in the smooth setting (i.e., a Riemannian manifold has dimension less or equal to $N$ and Ricci curvature greater or equal to $K$ if and only if it is a $\CD(K,N)$-space), it  is stable  under measured Gromov-Hausdorff convergence, and it implies fundamental properties as the Bishop-Gromov volume growth, Bonnet-Myers diameter bound, the Lichnerowicz spectral gap, the Brunn-Minkowski inequality, etc.
\\On the other hand, some basic properties like the local-to-global and the tensorization are not clear for the $\CD(K,N)$ condition. In order to remedy to this inconvenient, Bacher-Sturm \cite{BS2010} introduced a (a priori) weaker notion of curvature called \emph{reduced curvature condition}, and denoted with $\CD^*(K,N)$, which satisfies the aforementioned missing properties and share  the same nice geometric features of $\CD(K,N)$ (but some of the inequalities may not have the optimal constant). For  more details about curvature conditions see Subsection \ref{Subsec:CB}. 

As a matter of facts,  both the $\CD(K,N)$ and $\CD^*(K,N)$ conditions include Finsler geometries \cite{Ohta}-\cite{Villani09}. In order to isolate the Riemannian-like structures, Ambrosio-Gigli-Savar\'e \cite{Ambrosio-Gigli-Savare11b}  (see also \cite{AmbrosioGigliMondinoRajala} for a simplification of the axiomatization and the extension to $\sigma$-finite measures) introduced the class of $\RCD(K,\infty)$-spaces. Such notion strengthens the $\CD(K,\infty)$ condition with the linearity of the heat flow (notice that on a smooth Finsler manifold, the $\RCD(K,\infty)$ property is equivalent to saying that the manifold is, in fact, Riemannian); as proved in \cite{GigliMondinoSavare}, the $\RCD(K,\infty)$ condition is also stable under measured Gromov-Hausdorff convergence. Next, we briefly recall the definition of heat flow in m.m. spaces.
  
\

First of all on a m.m.s. $(X,\sfd,\mm)$ we cannot speak of differential (or gradient) of a function $f$ but at least the modulus of the differential is $\mm$-a.e. well defined, it is called weak upper differential and it is denoted with $|Df|_w$ (see Subsection \ref{Subsec:Calculus} for more details). With this object one defines  the Cheeger energy of a measurable function $f:X \to \R$ as
\begin{equation}\label{CE}
\Ch(f)= \begin{cases}
\frac 1 2 \int_X |D f|_w^2 \, \d \mm,\ \ \ \text{if}\ |D f|_w \in L^2(X,\mm),
\\
+\infty,\ \ \ \ \text{otherwise}.
\end{cases}
\end{equation}
 Since $\Ch$ is convex and lowersemicontinuous on $L^2(X,\mm)$, one can apply the classical theory of gradient flows of  convex functionals in Hilbert spaces \cite{Ambrosio-Gigli-Savare08} and define the heat flow $\H_t$ as the unique $L^2$-gradient flow of $\Ch$. The infinitesimal generator of this semigroup is called Laplacian and it is denoted with $\Delta$. Let us remark that in general $\Delta$ is \emph{not} a linear operator, and it is linear if and only if the heat flow $\H_t$ is linear.

\

In order to keep track of all the three conditions (lower bound on the Ricci curvature, finite upper bound on the dimension, and infinitesimal Riemannian-like behavior) Erbar, Kuwada and Sturm \cite{Erbar-Kuwada-Sturm13} and (slightly later, with different techniques) Ambrosio, Savar\'e and the second named author \cite{AMS}, introduced the class  $\RCD^*(K,N)$. Such class consists of those  m.m. spaces which satisfy the $\CD^*(K,N)$ condition and have linear heat flow. Also the $\RCD^*(K,N)$ condition is stable under measured Gromov-Hausdorff convergence, so that limit spaces of Riemannian manifolds with Ricci curvature bounded below by $K$ and dimension bounded above by $N$ are $\RCD^*(K,N)$-spaces. One of the main achievements of both groups of authors is that the $\RCD^*(K,N)$ condition is equivalent to the dimensional Bochner inequality
\begin{equation}\label{eq:DBI}
\Delta\frac{|\nabla f|^2}{2}\geq\frac{(\Delta f)^2}{N}+\nabla f\cdot\nabla \Delta f+K|\nabla f|^2,
\end{equation}
properly understood in a weak sense. Let us remark that a very useful property of the Bochner inequality proved by Savar\'e \cite{Savare2013} is that it self-improves (for more details see Subsection  \ref{Subsec:CB}).

$\RCD^*(K,N)$-spaces satisfy nice \emph{geometric} properties as the Cheeger-Gromoll splitting Theorem \cite{Gigli13}, the Laplacian comparison Theorem \cite{Gigli12}, the Abresh-Gromoll inequality \cite{GigliMosconi}; moreover, the local blow up for $\mm$-a.e. point is Euclidean \cite{GMR13} (more precisely, the space of local blow ups in a point contains a Euclidean space) and the geodesics are essentially non branching \cite{RajalaSturm}.

\

The main objective of this paper is instead  to  investigate special \emph{analytic} properties of $\RCD^*(K,N)$-spaces. More precisely, we wish to  prove estimates on the heat flow involving the lower bound $K$ on the curvature and the upper bound $N$ on the dimension. Our strategy is to use the dimensional Bochner inequality \eqref{eq:DBI} in combination with the $\Gamma$-calculus developed by  Bakry-Emery \cite{BE83} and Bakry-Ledoux \cite{BL2006} in the smooth setting. We have been inspired by the paper \cite{BG2011} of Baudoin and the first named author in which, in the Riemannian setting, a purely analytical approach to the Li-Yau program is provided. Such approach is flexible enough that can be adapted to the setting of m. m. spaces. A key role is also played by the fundamental papers of Ambrosio-Gigli-Savar\'e \cite{Ambrosio-Gigli-Savare11}-\cite{Ambrosio-Gigli-Savare11b}-\cite{AmbrosioGigliSavare12}-\cite{AGSBaEm}, where the calculus and the fine properties of the heat flow in m.m. spaces are investigated.

Before stating the main theorems let us recall that  $\probt X$ denotes the class of probability measures with finite second moment on $(X,\sfd)$; moreover, given a nonnegative Borel measure $\nn$ on $X$ and a measurable function $f:X\to \R$, $|Df|_{w,\nn}$ denotes the weak upper differential of $f$ with respect to $\nn$ (see Subsection \ref{Subsec:Calculus} for more details).   

Our first main result is the following generalization of the Li-Yau inequality.

\begin{theorem}[Li-Yau inequality]\label{thm:Li-Yau}
Let $(X,\sfd,\mm)$ be a $\RCD^*(0,N)$-space with $\mm(X)=1$ and let $f\in L^1(X,\mm)$ with $f \geq 0$ $\mm$-a.e. Then, for every $T>0$ one has
\begin{equation}\label{eq:LiYau1}
|D \H_T f|_w ^2 - (\Delta \H_T f) (\H_T f) \leq \frac{N}{2T} (\H_T f)^2 \quad \mm\text{-a.e.} \quad .
\end{equation}
If moreover $f \mm \in \probt X$, then the inequality above can be rewritten as 
\begin{equation}\label{eq:LiYau2}
|D \log \H_T f|_{w, (\H_T f) \mm}^2- \frac{d}{dt}_{|_{t=T}} \log \H_t f \leq \frac{N}{2T} \quad (\H_t f)  \mm\text{-a.e.}, 
\end{equation}
where $|\cdot|_{w, (\H_T f) \mm}$ denotes the weak upper differential with respect to the reference measure $(\H_T f) \mm$.
\end{theorem}

The second main result of the paper is a generalization to the setting of $\RCD^*(K,N)$-spaces of an inequality which was originally proved in the smooth setting by Bakry and Qian in \cite{BQ1999}. 
 
\begin{theorem}[Bakry-Qian inequality]\label{thm:Bakry-Qian}
Let $(X,\sfd,\mm)$ be a $\RCD^*(K,N)$-space with $\mm(X)=1$  for some $K>0$. Then, for every  $T>0$ and every $f\in L^1(X,\mm)$ with $f \geq 0$ $\mm$-a.e.  one has 
\begin{equation}\label{eq:BQ1}
\Delta \H_T f  \leq \frac{N K}{4} \, \H_T f \quad \mm\text{-a.e. } \quad.
\end{equation}
Of course, if we choose  the continuous representatives of  $\Delta \H_T f$ and $\H_T f$, then  the estimate \eqref{eq:BQ1} holds true for every $x \in X$.
\end{theorem}

In out third main result we extend to the setting of $\RCD^*(K,N)$-spaces an inequality which in the smooth setting was obtained in \cite{BG2011} by Baudoin and the first named author. Such inequality will be crucial for obtaining an Harnack inequality for the heat flow.

\begin{theorem}\label{thm:BG}
Let $(X,\sfd,\mm)$ be a $\RCD^*(K,N)$-space with $\mm(X)=1$ and  let $f\in L^1(X,\mm)$ with $f \geq 0$ $\mm$-a.e. Then, for every $T>0$ one has
\begin{equation}\label{eq:BG}
|D \H_T f|_w ^2   \leq e^{-\frac{2 K T}{3}} (\Delta \H_T f) \, \H_T f + \frac{N K}{3} \frac{e^{-\frac{4 K T}{3}}}{1-e^{-\frac{2KT}{3}}} (\H_T f)^2 \quad \mm\text{-a.e.}
\end{equation}
If moreover $f \mm \in \probt X$, then the inequality \eqref{eq:BG} can be rewritten as 
\begin{equation}\label{eq:BG2}
|D \log \H_T f|_{w, (\H_T f) \mm}^2 \leq   e^{-\frac{2 K T}{3}} \frac{\Delta \H_T f} {\H_T f} + \frac{N K}{3} \frac{e^{-\frac{4 K T}{3}}}{1-e^{-\frac{2KT}{3}}}  \quad  \quad (\H_t f)  \mm\text{-a.e.}, 
\end{equation}
where $|\cdot|_{w, (\H_T f) \mm}$ denotes the weak upper differential with respect to the reference measure $(\H_T f) \mm$.
\end{theorem}

The fourth result is an Harnack inequality for the heat flow. Let us remark that while the proof of the previous results was a (non trivial) adaptation of the proofs in the smooth setting mainly from \cite{BL2006}, \cite{BQ1999} and  \cite{BG2011}, the proof of the Harnack inequality uses new ideas from optimal transportation. Indeed the problem in adapting the smooth proofs is the (a priori)  lack of continuity of $|D \H_t f|_w$, in particular it is not clear if its restriction to a fixed geodesic makes sense. To overcome this difficulty we work with  families of geodesics where some optimal transportation is performed and, thanks also to the construction of good geodesics under curvature bounds by Rajala \cite{R2012}, we manage to prove the following theorem.

\begin{theorem}[Harnack inequality]\label{thm:Harnack}
Let $(X,\sfd,\mm)$ be a $\RCD^*(K,N)$-space with $\mm(X)=1$, and  let $f\in L^1(X,\mm)$ with $f \geq 0$ $\mm$-a.e. If $K\geq 0$, then for every $x,y \in X$ and $0<s<t$ we have  
\begin{equation}\label{eq:HarnackK>0}
(\H_t f)(y)\geq (\H_sf)(x) \; e^{- \frac{\sfd^2(x,y)}{4 (t-s)e^{\frac{2 Ks}{3}}}}    \left(\frac{1-e^{\frac{2 K}{3} s}}{1-e^{\frac{2 K}{3} t}} \right)^\frac{N}{2}.
\end{equation}  
If instead $K<0$, then
\begin{equation}\label{eq:HarnackK>0}
(\H_t f)(y)\geq (\H_sf)(x) \; e^{- \frac{\sfd^2(x,y)}{4 (t-s)e^{\frac{2 Kt}{3}}}}    \left(\frac{1-e^{\frac{2 K}{3} s}}{1-e^{\frac{2 K}{3} t}} \right)^\frac{N}{2}.
\end{equation}
\end{theorem}

We conclude observing that the inequalities above can be applied to the heat flow starting from a Dirac delta $\delta_x$, the so called heat kernel. Indeed, thanks to \cite[Subsection 6.1]{Ambrosio-Gigli-Savare11b} (see Subsection \ref{subsec:ImprovedRegHeat} for a brief summary), in $\RCD^*(K,N)$-spaces for every $x \in X$ one can define $\H_t \delta_x$; this is an absolutely continuous probability measure with Lipschitz density $p(x,y,t)$ which is non negative and  symmetric in $x$ and $y$. Applying the theorems above to $p$ we obtain the following corollary.

\begin{corollary}[Li-Yau and Harnack type estimates of the heat kernel]
Let $(X,\sfd,\mm)$ be a $\RCD^*(K,N)$-space. Then, the heat kernel $p$ defined above satisfies the following inequalities:
\begin{itemize}
\item[i)] (Li-Yau) If $K=0$, then for every $t>0$ one has
\begin{equation}\label{eq:LiYaup2}
|D \log p(t,x,\cdot)|_{w, \H_t \delta_x}^2- \frac{d}{dt} \log p(t,x,\cdot) \leq \frac{N}{2t} \quad \H_t\delta_x\text{-a.e.}, 
\end{equation}
where $|\cdot|_{w, \H_t \delta_x}$ denotes the weak upper differential with respect to the reference measure $H_t \delta_x$.

\item[ii)] (Bakry-Qian) If $K>0$, then for every  $t>0$  one has 
\begin{equation}\label{eq:BQp1}
\Delta p(t,x,\cdot)  \leq \frac{N K}{4} \, p(t,x,\cdot) \quad \mm\text{-a.e.}
\end{equation}

\item[iii)] (Baudoin-Garofalo) For every $t>0$ one has 
\begin{equation}\label{eq:BGp2}
|D \log p(t,x,\cdot)|_{w, \H_t(\delta_x)}^2 \leq   e^{-\frac{2 K t}{3}} \frac{\Delta p(t,x,\cdot)} {p(t,x,\cdot)} + \frac{N K}{3} \frac{e^{-\frac{4 K t}{3}}}{1-e^{-\frac{2Kt}{3}}}  \quad  \H_t(\delta_x)\text{-a.e.;} 
\end{equation}

\item[iv)] (Harnack) If $K\geq 0$ then for every $x,y \in X$ and $0<s<t$ it holds 
\begin{equation}\label{eq:HarnackpK>0}
p(t,y,z)\geq p(t,x,z) \; e^{- \frac{\sfd^2(x,y)}{4 (t-s)e^{\frac{2 Ks}{3}}}}    \left(\frac{1-e^{\frac{2 K}{3} s}}{1-e^{\frac{2 K}{3} t}} \right)^\frac{N}{2}   \quad.
\end{equation}  
If instead $K<0$ then
\begin{equation}\label{eq:HarnackpK>0}
p(t,y,z)\geq p(t,x,z) \; e^{- \frac{\sfd^2(x,y)}{4 (t-s)e^{\frac{2 Kt}{3}}}}    \left(\frac{1-e^{\frac{2 K}{3} s}}{1-e^{\frac{2 K}{3} t}} \right)^\frac{N}{2}   \quad.
\end{equation}
\end{itemize}
\end{corollary}

\smallskip
\noindent {\bf Acknowledgment:} The work on this paper begun at the workshop: ``Analysis on metric spaces'', IPAM, UCLA, March 18-22, 2013, and  was nearly completed during the workshop: ``Heat Kernels, Stochastic Processes and Functional Inequalities, Oberwolfach, MFO, May 5-11, 2013. We thank the organizers of both workshops for the gracious invitation and the IPAM and the MFO for the hospitality and the truly remarkable working atmosphere. We also gratefully acknowledge helpful discussions with Luigi Ambrosio and Giuseppe Savar\'e. 

\section{Preliminaries}\label{sec:preli}

\subsection{Calculus, Sobolev spaces and heat flow in metric measure spaces}\label{Subsec:Calculus}

Throughout the paper $(X,\sfd,\mm)$ will be a metric measure space, m.m.s. for short, i.e. $(X,\sfd)$ is a complete and separable metric space and $\mm$ is a non negative Borel measure. Even if some of the statements of this paper hold in case $\mm$ is a sigma finite measure, for simplicity we will always assume $\mm(X)=1$ and $\supp(\mm)=X$.
 
The heat flow and the calculus in a m.m.s have been the object of a series of papers of Ambrosio, Gigli and  Savar\'e  (see \cite{Ambrosio-Gigli-Savare11}, \cite{Ambrosio-Gigli-Savare11b} and \cite{AGSBaEm}); here, we briefly recall some useful facts. For more details the interested reader is referred to the aforementioned articles.

\

Let us start with some basic notations. We shall denote by $\LIP(X)$ the space
of Lipschitz functions,  by $\prob X$ the space of Borel probability measures on the complete and separable metric space $(X,\sfd)$ and by $\probt X \subset \prob X$ the subspace consisting of all the probability measures with finite second moment. Given an open interval $J\subset\R$, an
exponent $p\in [1,\infty]$ and $\gamma:J\to X$, we say that
$\gamma$ belongs to $AC^p(J;X)$ if
$$
\sfd(\gamma_s,\gamma_t)\leq\int_s^t g(r)\,\d r\qquad\forall s,\,t\in
J,\,\,s<t
$$
for some $g\in L^p(J)$. The case $p=1$ corresponds to
\emph{absolutely continuous} curves. It turns out that, if $\gamma$ belongs to $AC^p(J;X)$, there is a 
minimal function $g$ with this property, called
\emph{metric derivative} and given for a.e. $t\in J$ by
$$
|\dot\gamma_t|:=\lim_{s\to t}\frac{\sfd(\gamma_s,\gamma_t)}{|s-t|}.
$$
See \cite[Theorem~1.1.2]{Ambrosio-Gigli-Savare08} for the simple
proof. We say that an absolutely continuous curve $\gamma_t$ has
\emph{constant speed} if $|\dot\gamma_t|$ is (equivalent to) a
constant, and it is a \emph{geodesic} if
\begin{equation}\label{defgeo}
\sfd(\gamma_s,\gamma_t)=|t-s|\sfd(\gamma_0,\gamma_1)\qquad\forall s,\,t\in
[0,1].
\end{equation}
$(X,\sfd)$ is said \emph{geodesic space} if for any $x_0,\,x_1\in
X$ there exists a (constant speed) geodesic $\gamma$ joining $x_0$ and $x_1$ (i.e. $\gamma_0=x_0$ and $\gamma_1=x_1$); all the metric spaces we will work with will be assumed to be geodesic. We will denote by $\geo(X)$ the space of all constant speed
geodesics $\gamma:[0,1]\to X$, namely $\gamma\in\geo(X)$ if
\eqref{defgeo} holds. 

{F}rom the measure-theoretic point of view, when considering measures
on $AC^p(J;X)$ (resp. $\geo(X)$), we shall consider
them as measures on the Polish space $C(J;X)$ endowed with the sup
norm, concentrated on the Borel set $AC^p(J;X)$ (resp. closed set $\geo(X)$). We shall also use the
notation $\e_t:C(J;X)\to X$, $t\in J$, for the evaluation map at time $t$, namely
$\e_t(\gamma):=\gamma_t$;  and $(\e_t)_\sharp: \pro(C(J;X)) \to \pro(X)$ for the induced push-forward map of measures.

We now recall the notions of test plan, weak upper differential, and Sobolev space with respect to a reference probability measure
$\nn$ on $X$ (which may differ from $\mm$).

\begin{definition}[Test plan]\label{def:TestPlan}
We say that $\ppi\in\prob{C([0,1];X)}$ is a \emph{test plan relative to} $\nn$ if:
\begin{itemize}
\item[(i)] $\ppi$ is concentrated on $AC^2([0,1];X)$ and the action of $\ppi$ is finite:
$$
{\cal A}(\ppi):=\int\int_0^1|\dot\gamma_t|^2\,\d t\,\d\ppi(\gamma)<\infty.
$$ 
\item[(ii)] There exists $C\geq 0$ such that $(\e_t)_\sharp\ppi\leq C\nn$ for all $t\in [0,1]$.
\end{itemize}
\end{definition}

The following definition is inspired by the classical concept  of upper differential introduced by Heinonen and Koskela \cite{Heinonen-Koskela98}, that we now illustrate.
A Borel function $G:X\to [0,\infty]$ is an upper differential of a Borel function $f:X\to\R$ if 
$$
|f(\gamma_b)-f(\gamma_a)|\leq \int_a^bG(\gamma_s)|\dot\gamma_s|\,\d s
$$
for any absolutely continuous curve $\gamma:[a,b]\to X$. Being the inequality invariant under reparametrization one can also
reduce to curves defined in $[0,1]$.

\begin{definition}[The space $\calW_\nn$ and weak upper gradients]\label{def:wug}
Let $f:X\to\R$, $G:X\to [0,\infty]$ be Borel functions. We say that $G$ is a weak upper differential of $f$ relative to $\nn$  if
$$
|f(\gamma_1)-f(\gamma_0)|\leq \int_0^1G(\gamma_s)|\dot\gamma_s|\,\d s\qquad\text{for $\ppi$-a.e. $\gamma$}
$$
for all test plans $\ppi$ relative to $\nn$. We write $f\in\calW_\nn$ if $f$ has a weak upper differential in $L^2(X,\nn)$. The weak upper differential
relative to $\nn$ with minimal $L^2(X,\nn)$ norm (the so-called minimal weak upper differential) will be denoted by $|D f|_{w,\nn}$. In case $\nn=\mm$ we will simply write $|Df|_w$ in place of $|Df|_{w,\mm}$.
\end{definition}

\begin{remark}[Sobolev regularity along curves]\label{rem:charaweakgrad}\label{rem:SobCurve}
{\rm
A consequence of $\calW_\nn$ regularity is (see \cite[Remark~4.10]{AmbrosioGigliSavare12}) the Sobolev property along curves, namely for any
test plan $\ppi$ relative to $\nn$ the function $t\mapsto f(\gamma_t)$ belongs to the Sobolev space $W^{1,1}(0,1)$ and
$$
\left|\frac{\d}{\d t}f(\gamma_t)\right|\leq |Df|_w(\gamma_t)\, |\dot\gamma_t|\qquad\text{a.e. in $(0,1)$}
$$ 
for $\ppi$-a.e. $\gamma$. Conversely, assume that $g$ is Borel nonnegative, that for any test plan $\ppi$ the map 
$t\mapsto f(\gamma_t)$ is $W^{1,1}(0,1)$ and that
$$
\left|\frac{\d}{\d t}f(\gamma_t)\right|\leq g(\gamma_t)|\dot\gamma_t|\qquad\text{a.e. in $(0,1)$}
$$ 
for $\ppi$-a.e. $\gamma$. Then, the fundamental theorem of calculus in $W^{1,1}(0,1)$ gives that $g$ is a weak upper differential of $f$.
}\end{remark}

Weak differentials share with classical differentials many features, in particular the chain rule
\cite[Proposition~5.14]{Ambrosio-Gigli-Savare11}
\begin{equation}\label{eq:chainrule}
|D \phi(f)|_{w,\nn}=\phi'(f)|D f|_{w,\nn}\qquad\text{$\nn$-a.e. in $X$}
\end{equation}
for all $\phi:\R\to\R$ Lipschitz and nondecreasing on an interval containing the image of $f$. 
By convention, as in the classical chain rule, $\phi'(f)$ is arbitrarily defined at all points
$x$ such that $\phi$ is not differentiable at $x$, taking into account the fact that
$|D f|_{w,\nn}=0$ $\nn$-a.e. on this set of points.

The following theorem concerning the change of reference measure will be used later in the paper, for the proof see \cite[Theorem 3.6]{AmbrosioGigliMondinoRajala}.

\begin{theorem}[Change of reference measure]\label{thm:change}
Assume that $\rho=g\mm\in\probt{X}$ with $g\in L^\infty(X,\mm)$ and $\Ch(\sqrt{g})<\infty$.
Then:
\begin{itemize}
\item[(a)] $f\in\calW$ and $|Df|_w\in
L^2(X,\rho)$ imply $f\in\calW_\rho$ and $|D
f|_{w,\rho}=|Df|_w$ $\rho$-a.e. in $X$;
\item[(b)] $\log g\in\calW_\rho$ and $|D\log g|_{w,\rho}=|D g|_w/g$ $\rho$-a.e. in $X$.    
\end{itemize}
\end{theorem}

As mentioned in the introduction, a  fundamental object is the \emph{Cheeger energy} defined for a measurable function $f:X \to \R$ as in \eqref{CE} above.
%$$\Ch(f)=\frac 1 2 \int_X |D f|_w^2 \, \d \mm \quad,$$
%in case $|D f|_w \in L^2(X,\mm)$ and we set $\Ch(f)=+\infty$ otherwise.  
The domain of the Cheeger energy in $L^2(X,\mm)$ is by definition the space of Sobolev functions $W^{1,2}(X,\sfd,\mm)$. Notice that,  endowed with the norm
$$\|f\|_{W^{1,2}}^2:= \|f\|_{L^2}+ 2 \Ch(f), $$
$W^{1,2}(X,\sfd,\mm)$
is a Banach space, but in general it is not a Hilbert space. If it is a Hilbert space then the m.m.s. $(X,\sfd,\mm)$ is said \emph{infinitesimally Hilbertian}; for instance a smooth Finsler manifold is infinitesimally Hilbertian if and only if it is actually a Riemannian manifold. Let us recall that infinitesimal Hilbertianity has proved to be a very useful assumption both from the analytic point of view (see for instance \cite{Ambrosio-Gigli-Savare11b}, \cite{Gigli12} and \cite{GM12}) and from the geometric one (for instance in \cite{MonAng} the second named author defined a notion of angle in such spaces). 
The powerful fact of infinitesimally Hilbertian spaces is that not only a weak notion of modulus of the differential is defined, but also a scalar product between weak differentials can be introduced. We refer to Section 4.3 in \cite{Ambrosio-Gigli-Savare11b} for more details. Here, we just recall some basic facts. The scalar product $Df \cdot Dg$ for $f,g \in D(\Ch)$ is defined as the limit in $L^1(X,\mm)$ as $\varepsilon \downarrow 0$ of 
$$
Df \cdot Dg = \lim_{\varepsilon \downarrow 0} \frac{1}{2\varepsilon} \left( |D(f+\varepsilon g)|_w^2 - |Df)|^2_w\right).
$$ 
Moreover, the map $D(\Ch)^2\ni (f,g)\mapsto Df \cdot Dg \in L^1(X,\mm)$ is bilinear, symmetric, and satisfies the Cauchy-Schwarz inequality
$$|Df \cdot Dg|\leq |Df|_w |Dg|_w.$$

A basic approximation result (see Theorem 6.2 in \cite{Ambrosio-Gigli-Savare11})  states that for $f \in L^2(X,\mm)$ the Cheeger energy can also be
obtained by a relaxation procedure:
$$\Ch(f)=  \inf \{\liminf_{n\to \infty} \frac 1 2 \int_X |D f_n|_w^2 \, \d \mm \}, $$
where the infimum is taken over all sequences of Lipschitz functions $(f_n)$ converging to $f$ in
$L^2(X,\mm)$ and where $|D f_n|$ denotes the local Lipschitz constant (called also slope). In particular, Lipschitz functions are dense in $W^{1,2}(X,\sfd,\mm)$.
It turns out that $\Ch$ is a convex and lowersemicontinuous functional on $L^2(X,\mm)$. Therefore, one can define the Laplacian $-\Delta f \in L^2(X,\mm)$ of a function $f\in W^{1,2}(X,\sfd,\mm)$ has the element of minimal $L^2$-norm in the subdifferential $\partial^- \Ch(f)$, provided the latter is non empty.
Observe that, in general, the Laplacian is a non linear operator and it is linear if and only if $(X,\sfd,\mm)$ is infinitesimally Hilbertian (see for instance \cite{Gigli12}). 
 
\
 
Applying the classical theory of gradient flows of  convex functionals in Hilbert spaces (see for instance \cite{Ambrosio-Gigli-Savare08} for a comprehensive presentation) one can study the gradient flow of $\Ch$ in the space $L^2(X,\mm)$. More precisely one obtains that for every $f \in L^2(X,\mm)$ there exists a continuous curve $(f_t)_{ \in [0,\infty)}$ in $L^2(X,\mm)$, locally absolutely continuous in $(0, \infty)$ with $f_0=f$ such that $\frac{d}{dt} f_t= \partial^- \Ch(f_t)$ for a.e. $t>0$. In fact we have
$$f_t \in D(\Delta) \quad \text{ and } \quad  \frac{\d^+}{\d t} f_t = \Delta f_t \quad, \quad \forall t>0. $$ 
This produces a semigroup $(\H_t)_{t\geq 0}$ on $L^2(X,\mm)$ defined by $\H_t f= f_t$, where $f_t$ is the unique $L^2$-gradient flow of $\Ch$.

An important property of the heat flow is the maximum (resp. minimum) principle, see   \cite[Theorem 4.16]{Ambrosio-Gigli-Savare11}: if$f\in L^2(X,\mm)$ satisfies $f \leq C$ $\mm$-a.e. (resp. $f \geq C$ $\mm$-a.e.), then also $\H_t f \leq C$ $\mm$-a.e.  (resp. $\H_t f \geq C$ $\mm$-a.e.) for all $t \geq 0$. Moreover the heat flow preserves the mass:  for every $f \in L^2(X,\mm)$ 
$$\int_X \H_t f \, \d \mm=\int_X f \, \d \mm, \quad \forall t \geq 0.$$ 

Recall also that if $\Ch$ is quadratic, or in other words $(X,\sfd,\mm)$ is infinitesimally Hilbertian, then ${\cal E}(f,f):=\Ch(f)$ is a strongly local Dirichlet form on $L^2(X,\mm)$ with domain $D({\cal E})=W^{1,2}(X,\sfd,\mm)$. In this case, $\H_t$ is a semigroup of  selfadjoint linear operators on $L^2(X,\mm)$ with the Laplacian $\Delta$ as generator. Moreover, for $f \in W^{1,2}(X,\sfd,\mm)$ and $g\in W^{1,2}(X,\sfd,\mm)\cap D(\Delta)$ we have the integration by parts formula
$$\int_X Df \cdot Dg \, \d \mm= -\int_X f \, \Delta g \, \d \mm.$$

%Finally let us recall that thanks to a celebrated Theorem of Cheeger (see Theorem 6.1 in \cite{Cheeger}) if $(X,\sfd,\mm)$ is a doubling and Poincar\'e space (as it is in our working framework thanks to the $\RCD^*(K,N)$ condition), given a locally Lipschitz function $f$ on $X$, its slope (also called local Lipschitz constant)  defined by
%$$|Df|(x_0):=\limsup_{x\to x_0} \frac{|f(x)-f(x_0)|}{\sfd{x,x_0}} \text{ if $x_0$ is not isolated, and 0 otherwise } $$ 
%coincides $\mm$-almost everywhere with the minimal weak upper gradient $|Df|_w$. 
   
\subsection{Lower Ricci curvature bounds}\label{Subsec:CB} 

In the sequel we briefly recall those basic definitions and properties of spaces with lower Ricci curvature bounds that we will need later on.

For $\mu_0,\mu_1 \in \probt X$ the quadratic transportation distance $W_2(\mu_0,\mu_1)$ is defined by
\begin{equation}\label{eq:Wdef}
  W_2^2(\mu_0,\mu_1) = \inf_\sggamma \int_X \sfd^2(x,y) \,\d\ggamma(x,y),
\end{equation}
where the infimum is taken over all $\ggamma \in \prob{X \times X}$ with $\mu_0$ and $\mu_1$ as the first and the second marginal.
Assuming the space $(X,\sfd)$ to be geodesic, also the space $(\probt X, W_2)$ is geodesic. It turns out that any geodesic $(\mu_t) \in \geo(\probt X)$ can be lifted to a measure $\ppi \in \prob{\geo(X)}$, so that $(\e_t)_\#\ppi = \mu_t$ for all $t \in [0,1]$. Given $\mu_0,\mu_1\in\probt X$, we denote by $\gopt(\mu_0,\mu_1)$ the space of all
$\ppi \in \prob{\geo(X)}$ for which $(\e_0,\e_1)_\#\ppi$ realizes the minimum in \eqref{eq:Wdef}. If $(X,\sfd)$ is geodesic, then the set $\gopt(\mu_0,\mu_1)$ is non-empty for any $\mu_0,\mu_1\in\probt X$.

We turn to the formulation of the $\CD^*(K,N)$ condition, coming from  \cite{BS2010}. We refer to this source also for a detailed discussion of the relation of the $\CD^*(K,N)$ with the $\CD(K,N)$ condition previously introduced by Lott-Villani \cite{Lott-Villani09} and Sturm \cite{Sturm06II} (for recent development about the relations between $\CD(K,N)$ and $\CD^*(K,N)$ see also \cite{Cavalletti-Sturm12} and \cite{Cavalletti12}). Here, we recall that $\CD(K,N)$ implies $\CD^*(K,N)$, and that $\CD^*(K,N)$ implies $\CD(K^*,N)$ for $K^*=\frac{K(N-1)}{N}$. 

Given $K \in \R$ and $N \in [1, \infty)$, we define the distortion coefficient $[0,1]\times\R^+\ni (t,\theta)\mapsto \sigma^{(t)}_{K,N}(\theta)$ as
\[
\sigma^{(t)}_{K,N}(\theta):=\left\{
\begin{array}{ll}
+\infty,&\qquad\textrm{ if }K\theta^2\geq N\pi^2,\\
\frac{\sin(t\theta\sqrt{K/N})}{\sin(\theta\sqrt{K/N})}&\qquad\textrm{ if }0<K\theta^2 <N\pi^2,\\
t&\qquad\textrm{ if }K\theta^2=0,\\
\frac{\sinh(t\theta\sqrt{K/N})}{\sinh(\theta\sqrt{K/N})}&\qquad\textrm{ if }K\theta^2 <0.
\end{array}
\right.
\]
\begin{definition}[Curvature dimension bounds]
Let $K \in \R$ and $ N\in[1,  \infty)$. We say that a m.m.s.  $(X,\sfd,\mm)$
 is a $\CD^*(K,N)$-space if for any two measures $\mu_0, \mu_1 \in \prob X$ with support  bounded and contained in $\supp(\mm)$ there
exists a measure $\ppi \in \gopt(\mu_0,\mu_1)$ such that for every $t \in [0,1]$
and $N' \geq  N$ we have
\begin{equation}\label{eq:CD-def}
-\int\rho_t^{1-\frac1{N'}}\,\d\mm\leq - \int \sigma^{(1-t)}_{K,N'}(\sfd(\gamma_0,\gamma_1))\rho_0^{-\frac1{N'}}+\sigma^{(t)}_{K,N'}(\sfd(\gamma_0,\gamma_1))\rho_1^{-\frac1{N'}}\,\d\ppi(\gamma)
\end{equation}
where for any $t\in[0,1]$ we  have written $(\e_t)_\sharp\ppi=\rho_t\mm+\mu_t^s$  with $\mu_t^s \perp \mm$. If in addition $(X,\sfd,\mm)$ is infinitesimally Hilbertian, then we say that it is an $\RCD^*(K,N)$-space.
\end{definition}

One of the main achievements of the work of Erbar-Kuwada-Sturm  \cite{Erbar-Kuwada-Sturm13} (and of the independent and slightly subsequent work \cite{AMS} of the second named author in collaboration with Ambrosio and Savar\'e) is the following theorem asserting  that the $\RCD^*(K,N)$ condition is equivalent to the dimensional Bochner inequality, called also $\BE(K,N)$ condition.

\begin{theorem}[$\RCD^*(K,N)$ is equivalent to $\BE(K,N)$]
Let $(X,\sfd,\mm)$ be an infinitesimally Hilbertian m.m.s. Then, $(X,\sfd,\mm)$ is a $\RCD^*(K,N)$-space if and only if for all $f \in D(\Delta)$ with $\Delta f \in W^{1,2}(X,\sfd,\mm)$ and all $\varphi \in D(\Delta)$ bounded and non-negative with $\Delta \varphi \in L^\infty(X,\mm)$  we have
$$\int \frac 1 2 \Delta \varphi \, |D f|_w^2 \, \d \mm- \int \varphi \, D(\Delta f)\cdot Df \, \d \mm \geq K \int \varphi |D f|_w^2 \, \d \mm + \frac{1}{N} \int \varphi (\Delta f)^2 \, \d \mm. $$
\end{theorem}

In \cite{Savare2013}, Savar\'e proved a very important self-improvement property of the $\BE(K,\infty)$ condition. His arguments (in particular  Lemma 3.2 in \cite{Savare2013}) applied to  the finite dimensional $\BE(K,N)$ above give the following theorem, which will be very useful in the sequel of the paper. Before stating it let us denote with $\M_\infty$ the set of the functions $u \in W^{1,2}(X,\sfd,\mm)\cap L^\infty(X,\mm)$ for which there exists a measure $\mu= \mu^+-\mu^-$ with  $\mu_{\pm}\in W^{1,2}(X,\sfd,\mm)'_+$, the positive dual space to the Sobolev functions, such that
$$-\int_{X} Du \cdot D\varphi \, \d \mm= \int_X \varphi \, \d \mu \quad \forall \varphi \in W^{1,2}(X,\sfd,\mm). $$ 
For every $u \in \M_\infty$ we set $\Delta^*u := \mu$.

\begin{theorem}[Self-improvement of $\BE(K,N)$]
An infinitesimally Hilbertian m.m.s. $(X,\sfd,\mm)$ is a  $\RCD^*(K,N)$-space if and only if the following holds: for every $f \in L^\infty(X)\cap \LIP(X)\cap D(\Delta)$ with $\Delta f \in W^{1,2}(X,\sfd,\mm)$ we have $|D f|_w^2 \in \M_\infty$ and
$$\frac{1}{2} \Delta^* |D f|_w^2- D f \cdot D (\Delta f) \geq K |D f|_w^2 \mm + \frac{1}{N} \left( \Delta f\right)^2 \mm.$$
 \end{theorem}

For every $f \in L^\infty(X)\cap \LIP(X)\cap D(\Delta)$ with $\Delta f \in W^{1,2}(X,\sfd,\mm)$ we denote with $\Gamma^*_2(f)$ the finite Borel measure
\begin{equation}\label{eq:ImprBE(K,N)}
\Gamma^*_2(f):=\frac{1}{2} \Delta^* |D f|_w^2- D f \cdot D (\Delta f).
\end{equation}
Analogously to Lemma 2.6 in \cite{Savare2013} (see also page 12 of the same paper)  $\Gamma^*_2(f)$ has finite total variation. 

Recall also that thanks to the Bishop-Gromov property proved by Lott-Villani \cite{Lott-Villani09}   and Sturm \cite{Sturm06II} for $\CD(K,N)$-spaces, and the proof of a weak local Poincar\'e inequality for  $\CD(K,N)$-spaces by Lott-Villani \cite{LV2007} and Rajala \cite{R2011}, the  $\RCD^*(K,N)$-spaces are doubling and Poincar\'e as well. 

\

We close this subsection by discussing the geodesic structure of $(\probt{X},W_2)$ (see
\cite[Theorem~2.10]{Ambrosio-Gigli11} or \cite{Lisini07}) and the existence of good geodesics in $\CD^*(K,N)$-spaces (see \cite{R2012}).
If $\mu_0, \,\mu_1\in\probt X$ are connected by a constant speed geodesic $\mu_t$ in $(\probt X, W_2)$,
then there exists $\ppi \in \prob{\geo(X)}$ with $(\e_t)_\sharp\ppi = \mu_t$ for all $t\in [0,1]$
and
\[
 W_2^2(\mu_s,\mu_t) = \int_{\geo(X)}\sfd^2(\gamma_s,\gamma_t)\,\d\ppi(\gamma)=
 (s-t)^2\int_{\geo(X)}\ell^2(\gamma)\,\d\ppi(\gamma)\qquad\forall s,\,t\in [0,1],
\]
where $\ell(\gamma)=\sfd(\gamma_0,\gamma_1)$ is the length of the geodesic $\gamma$.
The collection of all the measures $\ppi$ with the above properties is denoted by
$\gopt(\mu,\nu)$.
The measure $\ppi$ is not uniquely determined by $\mu_t$, unless $(X,\sfd)$ is non-branching (the uniqueness of the lifting $\ppi$ in $\RCD^*(K,N)$-spaces is ensured   by \cite{Gigli12b} and \cite{RajalaSturm}), while
the relation between optimal geodesic plans and optimal Kantorovich plans is given by the fact that
$\gamma:=(\e_0,\e_1)_\sharp\ppi$ is optimal whenever $\ppi\in\gopt(\mu,\nu)$. We conclude by recalling a result of Rajala \cite[Thorem 1.2]{R2012} that we will use in the sequel.
\begin{theorem}[Improved Geodesics in $\CD^*(K,N)$-spaces] \label{thm:ImprGeod}
Let $(X, \sfd, \mm)$ be a $\CD^*(K,N)$-space for some $K\in\R$ and $N\in(1,\infty)$. Then for every couple of absolutely continuous probability measures  $\mu_0=\rho_0 \mm$, $\mu_1= \rho_1 \mm $  with bounded densities and bounded supports there exists $\ppi \in \gopt(\mu_0,\mu_1)$ such that
\\1)  $\ppi$ is a test plan in the sense of Definition \ref{def:TestPlan}; more precisely, called  $D = \diam (\supp(\mu_0) \cup \supp(\mu_1)) < \infty$ and 
 $\rho_t \mm:=\mu_t:=(\e_t)_\sharp (\ppi)$,  one has the density upper-bound
$$ \|\rho_t\|_{L^\infty(X,\mm)} \leq e^{\sqrt{K-N}D} \max\{\|\rho_0\|_{L^\infty(X,\mm)}, \|\rho_1\|_{L^\infty(X,\mm)}\}\quad .$$
\\2) $(\mu_t)$ satisfies the convexity property \eqref{eq:CD-def}.
\end{theorem}
Actually regarding the second statement, Rajala proves the  stronger assertion that the convexity property \eqref{eq:CD-def} holds for all triple of times $0\leq t_1<t_2<t_3\leq 1$ but we will not need this stronger version.

\subsection{Improved regularity of the heat flow in $\RCD^*(K,N)$-spaces}\label{subsec:ImprovedRegHeat}
Thanks to the identification, in $\RCD^*(K,\infty)$-spaces, of the heat flow $\H_t$ in $L^2(X,\mm)$ with the gradient flow ${\cal H}_t$ of the Shannon entropy functional  in the Wasserstein space, in \cite{Ambrosio-Gigli-Savare11b} several regularity properties of $\H_t$ have been deduced. We recall some of them. 
\\When $f \in  L^\infty(X,\mm)$, $\H_t f$ has a continuous representative, denoted by $\tilde{\H}_t f$, which is defined as follows (see Theorem 6.1 in \cite{Ambrosio-Gigli-Savare11b})
\begin{equation}\label{eq:ContRepHt}
\tilde{\H}_t f:=\int_X f \, \d {\cal H}_t(\delta_x).
\end{equation}
Moreover, for each $f \in L^\infty$ the map $(t,x)\mapsto \tilde{\H}_t f(x)$ belongs to $C_b((0,\infty)\times X)$. According to Theorem 6.8 in \cite{Ambrosio-Gigli-Savare11b}, for any $f \in L^\infty(X,\mm)$ we even obtain that  $\tilde{\H}_t f$ is Lipschitz. Finally, the classical Bakry-\'Emery gradient estimate holds (see Theorem 6.2 in \cite{Ambrosio-Gigli-Savare11b})
\begin{equation}\label{eq:BEflow}
|D \H_t f|_w^2 \leq e^{-2Kt} \H_t(|D f|_w^2) \quad \mm\text{-a.e.}
\end{equation}

We stress that all the previous results were established without any upper bound on the dimension. In case of finite dimension one obtains finer properties. For instance, if $(X,\sfd,\mm)$ is a $\RCD^*(K,N)$-space, then $\tilde{\H}_t f$ is Lipschitz and bounded for any $f \in L^1(X,\mm)$; indeed, thanks to Remark 6.4 in \cite{Ambrosio-Gigli-Savare11b}, keeping in mind   that $\RCD^*(K,N)$-spaces are doubling and Poincar\'e, one can show that the semigroup $\H_t$ is regularizing from $L^1(X,\mm)$ to $L^\infty(X,\mm)$.

Let us also recall that thanks to the self adjointness of $\Delta$ in $L^2(X,\mm)$ and the continuity of $\H_t$ as a map of $L^p(X,\mm)$ into itself for every $t\geq 0$ and every $p \in [1,\infty]$, we can apply the classical theory developed by Stein (see Theorem 1 in Chapter III of \cite{Stein}) and infer that $\H_t$ is an analytic semigroup in $L^p(X,\mm)$ for every $p \in (1,\infty)$; more precisely the map $t \mapsto \H_t$ has an analytic extension in the sense that it extends to an analytic $L^p(X,\mm)$-operator-valued function $t+i\tau\mapsto \H_{t+i \tau}$ defined in the sector of the complex plane
 \[
|\arg(t+i \tau)|<\frac{\pi}{2}\left(1-\left|\frac{2}{p}-1 \right| \right).
\]
Observe also that, since by assumption $W^{1,2}(X,\sfd,\mm)$ is a Hilbert space, and $\Ch$ is a convex and continuous functional on $W^{1,2}(X,\sfd,\mm)$, then it admits a unique gradient flow which coincides with the heat flow. If follows that, for every $f\in L^1(X,\mm)$, $t\mapsto \H_t f$ is a locally absolutely continuous curve on $(0,\infty)$ with values in $W^{1,2}(X,\sfd,\mm)$.
\\From the classical theory of semigroups, for every  $f\in D(\Delta)$ one has 
$$\Delta(\H_t f)=\H_t(\Delta f).$$ 
It follows, in particular, that $\Delta (\H_t f) \in W^{1,2}(X,\sfd,\mm)$ for every $t>0$.

Finally in their  recent paper \cite{Erbar-Kuwada-Sturm13}, Erbar-Kuwada-Sturm proved the dimensional Bakry-\'Emery $L^2$-gradient-Laplacian estimate: if $(X,\sfd,\mm)$ is a $\RCD^*(K,N)$-space, then for every $f \in D(\Ch)$ and every $t>0$, one has
$$|D \H_t f|_w^2+\frac{4Kt^2}{N(e^{2Kt}-1)} |\triangle \H_t f|^2 \leq e^{-2Kt} \H_t\left(|Df|_w^2\right).$$

\section{Two fundamental Lemmas}\label{sec:2lemmas}
Throughout the remainder of the paper $(X,\sfd,\mm)$ will be a $\RCD^*(K,N)$-space, for some $N\geq 1$ and $K \in \R$, with $\mm(X)=1$.

First of all, observe that given $f \in L^1(X,\mm)$ with $f\geq \delta>0$ $\mm$-a.e., thanks to the discussion of Subsection \ref{subsec:ImprovedRegHeat}, we already know that $\H_t f \in \LIP(X)$ and $\H_t f\geq \delta$ for every $t>0$; therefore the function $(\H_{T-t} f) \, |D(\log \H_{T-t} f)|^2_w$ is an element of $L^\infty(X,\mm)$ and, for every $T>0$ and $t\in [0,T]$, we can define
\begin{equation}\label{eq:defPhi(t)}
\Phi(t):= \H_t\left( (\H_{T-t} f)  |D(\log \H_{T-t} f)|^2_{w} \right).
\end{equation}
Notice that, for every $t\in[0,T)$, $\Phi(t) \in \LIP(X)$.

Secondly, notice that, given $f\in L^1(X,\mm)$ with $f\geq 0$ and $f \mm \in \probt X$, from the energy dissipation rate (see \cite{Ambrosio-Gigli-Savare11} and \cite{Ambrosio-Gigli-Savare11b}, in particular the estimate (6.2) of the latter) of the Shannon entropy $\int \rho \log \rho \, \d \mm$  and of the Fisher information ${\rm F}(\rho):=8\Ch(\sqrt{\rho})$ along the heat flow we obtain that $(\H_t f) \mm \in \probt X$, $\H_t f \log \H_t f \in L^1(X,\mm)$, and $|D \sqrt{\H_t f}|_w \in L^2(X,\mm)$.   Then, Theorem \ref{thm:change} implies that $\log \H_t f \in \calW_{(\H_t f)\mm}$ the weighted Sobolev space, and
\begin{equation}
|D \log(\H_t f)|_{w,(\H_t f)\mm} = \frac{|D (\H_t f)|_w}{\H_t f} \quad (\H_t f) \mm \text{-a.e.} 
\end{equation}
This last observation will be simply used to write the Li-Yau and Bakry-Qian inequalities in a compact form. 

%{\color{red}{As noticed above one can define $\Phi(t)$ in case $\mm(X)=\infty$; instead for definining $\Gamma^*_2(\log(\H_{T-t} f))$ we need that $\log(\H_{T-t} f\in L^\infty\cap \LIP(X) \ldots$ which is not reasonable if $H_{T-t} f \in L^1(X,\mm)$ and $\mm(X)=\infty$. So from now on we assume that $\mm(X)=1$ and $\frac 1 C \leq f \leq C$. In the end we will let $C\uparrow +\infty$}}

%The structure of our arguments is strongly inspired by the alternative proof of Li-Yau type inequalities by Baudoin-Garofalo \cite{BG2011} which avoids Jacobi fields.

\begin{lemma}\label{lemma1}
Let $(X,\sfd,\mm)$ be an $\RCD^*(K,N)$-space with $\mm(X)=1$, and let $f \in L^1(X,\mm)$ with $f\geq \delta>0$ $\mm$-a.e. For $0<t<T$ let $\Phi(t)$ be defined in \eqref{eq:defPhi(t)}. Then, for every $\varphi \in L^1(X,\mm)$ the map $[0,T] \to \R$ defined as $t \mapsto \int_X \Phi(t) \, \varphi \, \d\mm$ is  absolutely continuous on $[0,T]$, and
\begin{equation}\label{der}
\frac{d}{dt} \int_X \Phi(t) \, \varphi \, \d\mm = 2 \int_X \H_{T-t}f \, \H_t \varphi \, \d \Gamma_2^*(\log \H_{T-t} f) \quad  \text{for a.e. } t \in [0,T]. 
\end{equation}
\end{lemma}

\begin{proof}
The absolute continuity of $t\mapsto \int \Phi(t) \, \varphi \, \d\mm $  follows  by the smoothness of $t\mapsto \H_t$ as a $L^p$-operator valued map for all $p\in(1,\infty)$, the Lipschitz regularization of the heat flow with the bound \eqref{eq:BEflow}, and the absolute continuity of $t\mapsto \H_t f$ as a curve with values in $W^{1,2}(X,\sfd,\mm)$ (see Subsection \ref{subsec:ImprovedRegHeat}).

We now prove \eqref{der}. Observe that since by minimum principle $\H_{T-t} f\geq \delta$, and moreover $\Delta \H_{T-t} f \in W^{1,2}(X,\sfd,\mm)$, we have by chain rule that the absolutely continuous curve  $[0,T]\mapsto L^1(X,\mm)$ defined by $t\mapsto|D(\log \H_{T-t} f)|_w^2$ satisfies 
\begin{equation}\label{eq:dtlog1}
\frac{d}{dt}|D(\log \H_{T-t} f)|_w^2=2 |D(\log \H_{T-t} f)|_w^2 \, \frac{\Delta  \H_{T-t} f}{ \H_{T-t} f}- 2 D(\log \H_{T-t} f) \, \frac{D(\Delta  \H_{T-t} f)}{ \H_{T-t} f},
\end{equation}
for a.e. $t\in[0,T]$. Using the self-adjointness of the heat flow $\H_t$, and the regularity in $t$ discussed above, for a.e. $t\in[0,T]$ we compute
\begin{eqnarray}
\frac{d}{dt}\int_X \Phi(t) \, \varphi \, \d \mm &=& \frac{d}{dt} \int_X \H_{T-t}f \, |D(\log \H_{T-t} f)|_w^2  \, \H_t\varphi \, \d \mm \nonumber \\
                                                &=& \int_X -\Delta \H_{T-t} f \, |D(\log \H_{T-t} f)|_w^2  \, \H_t\varphi \, \d \mm     \nonumber \\
																								&&  + \int_X  \H_{T-t} f \, \frac{d}{dt}|D(\log \H_{T-t} f)|_w^2  \, \H_t\varphi \, \d \mm    \nonumber \\
																								&&  + \int_X  \H_{T-t} f \, |D(\log \H_{T-t} f)|_w^2  \, \Delta \H_t\varphi \, \d \mm  \nonumber \\
																								&=&  2\int_X 	 |D(\log \H_{T-t} f)|_w^2 \, \Delta  \H_{T-t} f \, \H_t \varphi \, \d \mm \nonumber \\
																								&& -2 \int_X D(\log \H_{T-t} f) \cdot D(\Delta  \H_{T-t} f) \, \H_t \varphi \, \d \mm \nonumber \\
																								&& +2 \int_X  D(\H_{T-t} f) \cdot D\left(|D\log \H_{T-t} f|_w^2\right)  \, \H_t \varphi \, \d \mm \nonumber \\
																								&& + \int_X  \H_{T-t} f \, \Delta^* \left(|D\log \H_{T-t} f|_w^2\right)  \, \H_t \varphi \, \d \mm,\label{eq:dtPhi1}
\end{eqnarray}
where in the last  equality we  used \eqref{eq:dtlog1} and integrated by parts the Laplacian in the forth row.

On the other hand, by the chain rule on $\Gamma^*_2$ we have 
\begin{eqnarray}
2 \Gamma_2^*(\log \H_{T-t} f)&=& \Delta^*\left(|D\log \H_{T-t} f|_w^2\right) -\frac{2}{\H_{T-t} f} D \log \H_{T-t} f \cdot D\left(\Delta  \H_{T-t} f \right) \mm \nonumber \\
                                           &&  +2 \frac{\Delta \H_{T-t} f}{\H_{T-t} f} \, |D \log \H_{T-t} f|_w^2 \mm +2 D \log \H_{T-t} f \cdot  D \left(|D \log \H_{T-t} f|_w^2\right) \mm \nonumber%\label{eq:gamma2log}.
\end{eqnarray}
Combining \eqref{eq:dtPhi1} and the last equation gives the thesis.
\end{proof}

The following proposition, which is based on  Lemma \ref{lemma1} above, generalizes an analogous result which, in the Riemannian case, was established in \cite{BG2011}. It will prove crucial for obtaining the Li-Yau type inequalities.

\begin{proposition}\label{prop2}
Let $(X,\sfd,\mm)$ be a $\RCD^*(K,N)$-space with $\mm(X)=1$, $f\in L^1(X,\mm)$ with $f \geq \delta>0$ $\mm$-a.e.,  and $\Phi$ defined as in \eqref{eq:defPhi(t)}. Let $a(\cdot) \in C^1([0,T],\R^+)$ be nonnegative function, and let $\gamma\in C([0,T],\R)$ be another real function.
Then, for every $\varphi \in L^1(X,\mm)$ with $\varphi\geq 0$ $\mm$-a.e., the function 
$$[0,T]\ni t \; \mapsto \; \int_X \Phi(t) a(t) \varphi \, \d \mm \in \R$$
 is absolutely continuous and for a.e.  $t \in [0,T]$ one has  
\begin{equation}\label{eq:prop2}
\frac{d}{dt} \int_X \Phi(t) a(t) \varphi \, \d \mm \geq \int_X\left[ \left(a'(t)-\frac{4 a(t) \gamma(t)}{N}+2 K a(t) \right) \Phi(t)+ \frac{4 a(t) \gamma(t)}{N} \Delta \H_T f - \frac{2 a(t) \gamma^2(t)}{N} \H_T f \right] \varphi \, \d \mm. 
\end{equation} 
\end{proposition}

\begin{proof}
Since by assumption $a(\cdot)$ is $C^1$, the regularity of the map $t \mapsto \int_X \Phi(t) a(t) \varphi \, \d \mm$ follows from Lemma  \ref{lemma1}. 
By applying Lemma \ref{lemma1} and the improved $\BE(K,N)$ condition \eqref{eq:ImprBE(K,N)} we obtain
\begin{eqnarray}
\frac{d}{dt} \int_X \Phi(t) a(t) \varphi \, \d 	\mm &=& \int_X \Phi(t) a'(t) \varphi \, \d \mm + 2 \int_X \H_{T-t} f \, \H_t(a(t)\varphi) \, \d\Gamma^*_2(\log \H_{T-t}f) \nonumber \\
                                                    &\geq &  \int_X \Phi(t) a'(t) \varphi \, \d \mm + 2 K \int_X \H_{T-t} f \, \H_t(a(t)\varphi) \, |D  \log \H_{T-t}f|_w^2 \, \d \mm \nonumber \\
												                            && + \frac{2}{N}	\int_X \H_{T-t} f \, \H_t(a(t)\varphi) \, \left(\Delta  \log \H_{T-t}f \right)^2 \, \d \mm.\label{eq:dtaPhi1}													
\end{eqnarray} 
Now observe that 
\begin{equation} \label{eq:dtaPhi2}
(\Delta \log \H_{T-t} f)^2 \geq 2 \gamma(t) \, \Delta (\log \H_{T-t} f) - \gamma(t)^2,
\end{equation}
and by chain rule
\begin{equation} \label{eq:dtaPhi3}
\Delta \log \H_{T-t} f = \frac{\Delta \H_{T-t} f}{\H_{T-t} f} - |D \log \H_{T-t} f|_w^2.
\end{equation}
The conclusion follows combining \eqref{eq:dtaPhi1}, \eqref{eq:dtaPhi2} and \eqref{eq:dtaPhi3}, keeping in mind that $\H_t(\Delta \H_{T-t} f)=\H_T \Delta f$  and the selfadjointness of the heat flow.
\end{proof}

\section{Proof of the main results}\label{sec:proof}
In order to obtain the desired Li-Yau type inequalities we make some appropriate choices in Proposition \ref{prop2}.
Let us take a function $a(\cdot)$ as in Proposition \ref{prop2} such that  $a(0)=1$ and $a(T)=0$, and $\gamma$ such that
\begin{equation}\label{eq:defgamma}
a'(t)-\frac{4 a(t) \gamma(t)}{N}+2 K a(t) \equiv 0
\end{equation}
i.e. $\gamma(t):=\frac{N}{4} \left(\frac{a'(t)}{a(t)}+2 K\right)$. Then, the following proposition holds.
\begin{proposition}\label{prop:PreLiYau}
Let $(X,\sfd,\mm)$ be a $\RCD^*(K,N)$-space with $\mm(X)=1$, and  $f\in L^1(X,\mm)$ with $f \geq \delta>0$ $\mm$-a.e. Fix $T>0$, and  let $a(\cdot) \in C^1([0,T], \R^+)$ with $a(0)=1$ and $a(T)=0$. Then, the following inequality holds $\mm$-a.e. :
\begin{equation}\label{eq:PreLIYAU}
|D \log \H_T f|_w^2 \leq \left(1- 2 K \int_0^T a(t) \, \d t \right) \frac{\Delta \H_T f}{ \H_T f}+ \frac{N}{2} \left(\int_0^T \frac{a'(t)^2}{4 a(t)}\, \d t- K+ K^2 \int_0^T a(t) \, \d t\right)   \; .
\end{equation}
 \end{proposition}

\begin{proof}
With $\gamma$ chosen as in  \eqref{eq:defgamma}, for every $\varphi \in L^1(X,\mm)$ with $\varphi \geq 0$ $\mm$-a.e., integrate \eqref{eq:prop2} in $t$ from 0 to $T$ in order to obtain the following inequality
\begin{eqnarray}
-\int_X \H_T f |D \log \H_T f|_w^2 \,\varphi \, \d \mm &\geq & \int_0^T \left( \int_X  \left(a'(t)+2 K a(t)\right) \Delta \H_T f \, \varphi \, \d \mm \right) \, \d t  \nonumber \\
&& - \frac{N}{2} \int_0^T \left( \int_X \left(\frac{a'(t)^2}{4 a(t)}+K a'(t)+K^2a(t) \right) \H_T f \, \varphi \, \d \mm \right) \, \d t \quad . \nonumber 
\end{eqnarray}
Using Fubini's Theorem in the right-hand side, and recalling the assumption on $a(\cdot)$,  we obtain
\begin{eqnarray}
-\int_X \H_T f |D \log \H_T f|_w^2 \,\varphi \, \d \mm &\geq& \int_X \bigg[ \left(1-2 K \int_0^T a(t)\, \d t \right)  \Delta \H_T f \nonumber \\
  &&\qquad  + \frac{N}{2} \left( \int_0^T \frac{a'(t)^2}{4 a(t)} \d t - K +K^2\int_0^T a(t) \, \d t \right)\H_T f \bigg]\varphi \, \d \mm \quad.\nonumber
\end{eqnarray}
Since the last inequality holds for every $\varphi \in L^1(X,\mm)$ with $\varphi \geq 0$ $\mm$-a.e., and since both the integrands are $L^\infty(X,\mm)$ functions, the conclusion follows.
\end{proof}

For what follows it is useful to perform a change of variable in \eqref{eq:PreLIYAU}. Namely, calling $V(t):=\sqrt{a(t)}$, with a straightforward computation we find 
\begin{equation}\label{eq:PreLIYAUV}
|D \log \H_T f|_w + \left(2 K \int_0^T V^2(t) \, \d t -1 \right) \frac{\Delta \H_T f}{ \H_T f} \leq  \frac{N}{2} \left(\int_0^T V'(t)^2\, \d t- K+ K^2 \int_0^T V(t)^2 \, \d t\right)   \;. 
\end{equation}

With a particular choice  of the function $V(\cdot)$ in \eqref{eq:PreLIYAUV} (see the proof in Subsection \ref{SS:LY} below),  the celebrated Li-Yau inequality stated in Theorem \ref{thm:Li-Yau} will easily follow.  

\subsection{Proof of the Li-Yau inequality, Theorem \ref{thm:Li-Yau}}\label{SS:LY}
Let $\varepsilon>0$, set $f_\varepsilon:=f+\varepsilon$ and notice that $f_\varepsilon\geq \varepsilon >0$ $\mm$-a.e. so that we can  apply \eqref{eq:PreLIYAUV}  to $f_\varepsilon$ and $K=0$, obtaining
\begin{equation} \label{eq:LiYauproof1}
|D \log \H_T f_\varepsilon|_w^2- \frac{\Delta \H_T f_\varepsilon}{\H_T f_\varepsilon} \leq \frac{N}{2} \int_0^T V'(t)^2 \, \d t  \quad \mm\text{-a.e.} 
\end{equation} 
Choosing $V(t):=1-\frac{t}{T}$ (notice that this choice minimizes the integral in the right hand side among all the $C^1([0,T], \R^+)$ functions null at $T$ and equal to $1$ at $0$), we obtain
\begin{equation} \label{eq:LiYauproof2}
|D \log \H_T f_\varepsilon|_w^2- \frac{\Delta \H_T f_\varepsilon}{\H_T f_\varepsilon} \leq \frac{N}{2T}  \quad \mm\text{-a.e.} \quad  .
\end{equation} 
Recalling that $\H_t \varepsilon=\varepsilon$, from the linearity of the weak differential and of the Laplacian we have
$$|D \log \H_T f_\varepsilon|_w= \frac{|D \H_T f|_w}{\H_t f+\varepsilon}\quad \text{and} \quad \Delta \log \H_T f_\varepsilon= \frac{\Delta \H_T f}{\H_t f+\varepsilon}, $$
which, substituted into \eqref{eq:LiYauproof2}, gives
\begin{equation} \label{eq:LiYauproof3}
|D  \H_T f|_w^2- (\Delta \H_T f) (\H_T f +\varepsilon) \leq \frac{N}{2T}  (\H_T f +\varepsilon)^2  \quad \mm\text{-a.e.}  .
\end{equation} 
Letting $\varepsilon \downarrow 0$ in \eqref{eq:LiYauproof3} gives \eqref{eq:LiYau1}. In order to obtain the second formulation \eqref{eq:LiYau2}, observe that if $f\mm \in\probt X$, from the discussion in the beginning of Section \ref{sec:2lemmas} we know that   $\log \H_T f \in {\cal S}^2_{(\H_T f)\mm}$ the weighted Sobolev space, and
\begin{equation}\label{eq:DlogHt}
|D \log(\H_T f)|_{w,(\H_T f)\mm} = \frac{|D (\H_T f)|_w}{\H_T f} \quad (\H_T f) \mm \text{-a.e.} 
\end{equation}
The estimate \eqref{eq:LiYau2} thus follows  combining  \eqref{eq:LiYau1} and \eqref{eq:DlogHt}.
\hfill$\Box$.

\subsection{Proof of Theorems \ref{thm:Bakry-Qian} and \ref{thm:BG}}

In this subsection we provide the proofs of Theorems \ref{thm:Bakry-Qian} and \ref{thm:BG}.
 
\begin{proof}[Proof of Theorem \ref{thm:Bakry-Qian}]
As in the proof of Theorem \ref{thm:Li-Yau}, for $\varepsilon>0$ we set $f_\varepsilon:=f+\varepsilon$ and we apply \eqref{eq:PreLIYAUV}  to $f_\varepsilon$ with $V(t):=1-\frac{t}{T}$. A straightforward computation gives for any $t>0$
\begin{equation}\nonumber % \label{eq:BQproof1}
\left(\frac{2Kt}{3}-1\right) \frac{\Delta \H_t f_\varepsilon}{\H_t f_\varepsilon} \leq \frac{N}{2} \left( \frac{1}{t}+\frac{K^2 t} {3}-K \right) \quad \mm\text{-a.e.} 
\end{equation} 
Since for $t\geq \frac{2}{K}$ the term $\frac{2Kt}{3}-1$ is strictly positive we obtain
\begin{equation}\nonumber % \label{eq:BQproof1}
\Delta \H_t f_\varepsilon \leq  \frac{\frac{N}{2} \left( \frac{1}{t}+\frac{K^2 t} {3}-K \right) } {\frac{2Kt}{3}-1} \H_t f_\varepsilon  \quad \mm\text{-a.e.}  
\end{equation} 
An easy computation shows that the fraction in the right hand side is bounded above by $\frac{NK}{4}$ if and only if $t \geq \frac{2}{K}$. Recalling that $\H_t f_{\varepsilon} = \H_t f+ \varepsilon$ and $\Delta (\H_t f_\varepsilon)= \Delta (\H_t f)$, by letting  $\varepsilon \downarrow 0$ we reach the desired conclusion. 

\end{proof}

%\newline
%\newline

\begin{proof}[Proof of Theorem \ref{thm:BG}]
Applying \eqref{eq:PreLIYAUV} to $f_\varepsilon:=f+\varepsilon$, for a fixed $\varepsilon>0$, and 
$$V(t):=\frac{e^{-\frac{Kt}{3}} \left( e^{-\frac{2Kt}{3}}- e^{-\frac{2KT}{3}} \right)}{1-e^{-\frac{2KT}{3}}}, $$
the proof can be performed analogously to the one of Theorem \ref{thm:Li-Yau}.

\end{proof}

%\hfill$\Box$ 

\subsection{Proof of Theorem \ref{thm:Harnack} }

In this subsection we  will use ideas from optimal trasport (which seem to have been used for the first time in this context), in combination with Theorem \ref{thm:BG} above, to prove Theorem \ref{thm:Harnack}. As in the previous proofs let $\fe:=f+\varepsilon$ for some $\varepsilon>0$. Applying \eqref{eq:BG2} above to $\fe$, we find 
\begin{equation}\label{eq:pfHar1}
-\frac{d}{dt} \log(\H_t \fe)\leq -e^{\frac{2 K t}{3}} |D \log \H_t \fe|_w^2  + \frac{N K}{3} \frac{e^{-\frac{2 K t}{3}}}{1-e^{-\frac{2Kt}{3}}} \quad \mm\text{-a.e.} 
\end{equation}
Recall that in our notation $\supp(\mm)=X$.
%{\color{red}PROBLEM: the inequality above holds just $\mm$-a.e.; this is not enough for our aim since later we want to consider the restriction of this inequality to a curve whose support has measure zero; notice also that the weak gradient is upper semicontinuous so the semicontinuity goes to the wrong direction. One should prove that $\Gamma{\H_t f}$ is continuous, I searched in the literature but this seems not clear at this moment, we should work on this and one week is not enough in my opinion.}
Fix $x,y \in X$ and  $r>0$ (in the end we will let $r\downarrow 0$), and set  
\[
z^r_0=\mm(B_r(y))^{-1},\ \ \ \ \ z^r_1=\mm(B_r(x))^{-1}.
\]  
Define $\mu^r_0,\mu^r_1 \in \probt X$ as 
$$\mu^r_0:=z^r_0  \, \chi_{B_r(y)} \quad \text{ and } \quad  \mu^r_0:=z^r_1 \, \chi_{B_r(x)},$$
where $\chi_E$ is the characteristic function of the subset $E$. 
\\Let $\ppi^r\in \gopt(\mu^r_0, \mu^r_q)$ be given by Theorem \ref{thm:ImprGeod} and recall that it is a test plan in the sense of Definition \ref{def:TestPlan}. 
%   let $\gamma$ be a constant speed geodesic from $y$ to $x$, i.e.
%\begin{equation}\label{eq:gamma}
%\gamma(0)=y, \; \gamma(1)=x,\;\text{and} \quad  |\dot{\gamma}|(t)=\sfd(x,y) \;\; {\cal L}^1\text{-a.e. } t \in [0,1] \quad ;
%\end{equation}
For any fixed $0<s<t$  define $\alpha: AC^2([0,1],X)\times [0,1]\to X \times [s,t]$ as 
\begin{equation}\label{def:alpha}
\alpha(\gamma, \tau):= (\gamma(\tau), t+\tau(s-t)).
\end{equation}
Let also $\ue(z,\tau):=\H_\tau \fe (z)$ be the spatial-continuous (i.e. in the variable $z$; actually it is even Lipschitz in $z$)   representative given by \eqref{eq:ContRepHt}, and set $\phi_\varepsilon(\gamma, \tau):= \log \ue(\alpha(\gamma, \tau))$. 
%Therefore for any 
%\begin{equation}\label{eq:defphi}
%\varphi \in L^1(X,\mm) \; \text{ with } \;  \varphi \geq 0 \; \mm\text{-a.e.} \; \text{ and } \int_X \varphi \, \d\mm=1 \quad,  
%\end{equation}
Using the chain rule and recalling Remark \ref{rem:SobCurve}, we have
\begin{eqnarray}
\int \log\left(\frac{\ue(\gamma_1,s)}{\ue(\gamma_0,t)}\right) \, \d \ppi^r(\gamma) &=&  \int \left( \int_0^1 \phi_\varepsilon'(\gamma,\tau) \d \tau \right) \d\ppi^r(\gamma) \nonumber \\
                                           &\leq& \int\left( \int_0^1 |D\log(\ue)|_w(\alpha(\gamma,\tau)) \, |\dot{\gamma}| \d \tau \right) \d \ppi^r(\gamma)  \nonumber \\
																					&&-(t-s) \int\left( \int_0^1 \left(\frac{\partial}{\partial t} \log(\ue)  \right) (\alpha(\gamma,\tau)) \, \d \tau \right) \d\ppi^r(\gamma) . \label{eq:pfHarn1b}
\end{eqnarray}
Since $\ppi^r$ is a test plan, \eqref{eq:pfHar1} implies that for $\ppi^r$-a.e. $\gamma$,  and every $\tau\in[0,1]$, one has
\begin{equation}\label{eq:pfHarn1c}
-\left( \frac{\partial}{\partial t} \log(\ue)\right) (\alpha(\gamma,\tau)) \leq -e^{\frac{2 K}{3} (t+\tau(s-t))} |D \log \ue|_w^2 (\alpha(\gamma,\tau)) + \frac{N K}{3} \frac{e^{-\frac{2 K}{3}(t+\tau(s-t)) }}{1-e^{-\frac{2K}{3}(t+\tau(s-t))}}. 
\end{equation}
 
Estimating the first addendum of \eqref{eq:pfHarn1b} with Cauchy-Schwarz inequality and  the second with  \eqref{eq:pfHarn1c} , for any $\eta>0$ to be fixed later, we find
\begin{eqnarray}
\int \log\left(\frac{\ue(\gamma_1,s)}{\ue(\gamma_0,t)}\right) \, \d \ppi^r(\gamma) &\leq& \frac{\eta}{2} \int\left( \int_0^1 |D\log \ue|_w^2(\alpha(\gamma,\tau)) \, \d \tau \right) \d\ppi^r(\gamma)   + \frac{1}{2\eta} \int |\dot{\gamma}|^2 \, \d \ppi^r(\gamma) \nonumber \\
                                             && -(t-s) \int \left( \int_0^1 e^{\frac{2 K}{3} (t+\tau(s-t))} |D \log  \ue|_w^2 (\alpha(\gamma, \tau)) \, \d \tau \right) \d \ppi^r(\gamma) \nonumber \\
																						&& +(t-s) \frac{NK}{3}  \int_0^1 \frac{e^{-\frac{2 K}{3} (t+\tau(s-t))}}{1-e^{-\frac{2 K}{3} (t+\tau(s-t))}}  \d \tau. \label{eq:pfHarn2}
\end{eqnarray}

\

CASE 1: $K\geq 0$. A direct computation shows that
\begin{equation}\label{eq:lasInt}
(t-s)\frac{NK}{3} \int_0^1 \frac{e^{-\frac{2 K}{3} (t+\tau(s-t))}}{1-e^{-\frac{2 K}{3} (t+\tau(s-t))}} \d \tau = \frac{N}{2} \log\left(\frac{1-e^{\frac{2 K}{3} t}}{1-e^{\frac{2 K}{3} s}} \right).
\end{equation}
Moreover, observing that the function $\tau \mapsto  e^{\frac{2 K}{3} (t+\tau(s-t))}$ is non increasing, we can estimate
\begin{equation}\label{eq:2line}
\int \left( \int_0^1 e^{\frac{2 K}{3} (t+\tau(s-t))} |D \log  \ue|_w^2 (\alpha(\gamma,\tau))\d \tau \right) \d \ppi^r(\gamma)\geq e^{\frac{2 Ks}{3}}  \int\left( \int_0^1|D \log  \ue|_w^2 (\alpha(\gamma,\tau)) \d \tau \right) \d \ppi^r(\gamma).
\end{equation}
Therefore, choosing $\eta:=2(t-s)e^{\frac{2 Ks}{3}}$, and substituting \eqref{eq:lasInt} and \eqref{eq:2line} into \eqref{eq:pfHarn2}, we obtain
\begin{equation}\label{eq:pfHarn3}
\int \log\left(\frac{\ue(\gamma_1,s)}{\ue(\gamma_0,t)}\right) \, \d \ppi^r(\gamma) \leq  \frac{1}{4 (t-s)e^{\frac{2 Ks}{3}}} \int |\dot{\gamma}|^2 \, \d \ppi^r(\gamma)+  \frac{N}{2} \log\left(\frac{1-e^{\frac{2 K}{3} t}}{1-e^{\frac{2 K}{3} s}} \right).
\end{equation}
Since by construction (for more details see also the last paragraph of Subsection \ref{Subsec:CB}) $\ppi^r$ is a probabililty measure  concentrated along (constant speed) geodesics connecting points of $B_r(y)$ to points of $B_r(x)$, then for $\ppi^r$-a.e. $\gamma$ we have $\gamma_0\in B_r(y)$ and $\gamma_1\in B_r(x)$; recalling that $\ue$ is continuous (actually it is even Lipschitz) in the spatial variable $z$, letting $r\downarrow 0^+$ we find
$$
\lim_{r\downarrow 0 } \int \log\left(\frac{\ue(\gamma_1,s)}{\ue(\gamma_0,t)}\right) \, \d \ppi^r(\gamma) = \log\left(\frac{\ue(x,s)}{\ue(y,t)}\right);
$$
and 
$$
\lim_{r\downarrow 0 }\int |\dot{\gamma}|^2 \, \d \ppi^r(\gamma)= \lim_{r\downarrow 0 }\int \sfd^2(\gamma_0,\gamma_1) \, \d \ppi^r(\gamma)= \sfd^2(y,x). 
$$
It follows that
\begin{equation}\nonumber
\log\left(\frac{\ue(x,s)}{\ue(y,t)}\right)  \leq  \frac{\sfd^2(x,y)}{4 (t-s)e^{\frac{2 Ks}{3}}} +  \frac{N}{2} \log\left(\frac{1-e^{\frac{2 K}{3} t}}{1-e^{\frac{2 K}{3} s}} \right),
\end{equation}
which is the sought for Harnack inequality for $\fe$. Letting $\varepsilon\downarrow 0$ we obtain the desired conclusion.

\

CASE 2: $K<0$. In this case  the function $\tau \mapsto  e^{\frac{2 K}{3} (t+\tau(s-t))}$ is non decreasing, so we can estimate
\begin{equation}\label{eq:2line2}
 \int\left(\int_0^1 e^{\frac{2 K}{3} (t+\tau(s-t))} |D \log  \ue|_w^2 (\alpha(\gamma,\tau))\d \tau \right) \d\ppi^r(\gamma) \geq e^{\frac{2 Kt}{3}} \int\left(\int_0^1|D \log  \ue|^2_w (\alpha(\gamma,\tau)) \d \tau \right) \d \ppi^r(\gamma).
\end{equation}
Therefore, choosing $\eta:=2(t-s)e^{\frac{2 Kt}{3}}$, substituting \eqref{eq:lasInt} and \eqref{eq:2line2} into \eqref{eq:pfHarn2}, and finally letting $r\downarrow 0$ as above we obtain
\begin{equation}\label{eq:HarnackK<0}
\log\left(\frac{\ue(x,s)}{\ue(y,t)} \right) \leq  \frac{\sfd^2(x,y)}{4 (t-s) \; e^{\frac{2 Kt}{3}}} +  \frac{N}{2} \log\left(\frac{1-e^{\frac{2 K}{3} t}}{1-e^{\frac{2 K}{3} s}} \right).
\end{equation}
Letting $\varepsilon\downarrow 0$ we reach the desired conclusion.
\hfill$\Box$

\def\cprime{$'$}

\end{document}